\documentclass{llncs}
\usepackage{amsmath}
\usepackage{amssymb}
\usepackage[usenames,dvipsnames,svgnames,table]{xcolor}
\usepackage[colorlinks=true, citecolor=blue, linkcolor=Crimson]{hyperref}

\usepackage{color, etoolbox, xspace}

\newcommand{\DNR}{\mathrm{DNR}}

\newcommand{\dom}{\mathrm{dom}}

\newcommand{\bstrings}{\omega^{<\omega}}
\newcommand{\dset}[2]{\{#1 : #2 \}}

\DeclareMathOperator{\converges}{\downarrow}

\DeclareMathOperator{\SNR}{SNR}
\DeclareMathOperator{\SNPR}{SNPR}
\DeclareMathOperator{\CI}{CI}
\newcommand{\la}{\langle}
\newcommand{\ra}{\rangle}

\begin{document}
\title{From eventually different functions to pandemic numberings\thanks{This work was partially supported by
	a grant from the
	Simons Foundation (\#315188 to Bj\o rn Kjos-Hanssen).
	}
}

\author{
	Achilles A. Beros\inst{1} \and
	Mushfeq Khan\inst{1} \and
	Bj{\o}rn Kjos-Hanssen\inst{1}\orcidID{0000-0002-6199-1755} \and
	Andr\'e Nies \inst{2}
}
\institute{
	University of Hawai`i at M{\=a}noa, Honolulu HI 96822, USA
	\email{\{beros,mushfeq,bjoernkh\}@hawaii.edu} \and
	University of Auckland, New Zealand
	\email{nies@cs.auckland.ac.nz}
}

	\maketitle

	\begin{abstract}
		A function is strongly non-recursive (SNR) if it is eventually different from each recursive function.
		We obtain hierarchy results for the mass problems associated with computing such functions with varying growth bounds.
		In particular, there is no least and no greatest Muchnik degree among those of the form $\SNR_f$ consisting of SNR functions bounded by varying recursive bounds $f$.

		We show that the connection between SNR functions and canonically immune sets is, in a sense, as strong as that between DNR (diagonally non-recursive) functions and effectively immune sets. Finally, we introduce pandemic numberings, a set-theoretic dual to immunity.
	\end{abstract}

	\section{Introduction}
		It has been known for over a decade that bounding diagonally non-recursive functions by various computable functions leads to a hierarchy of computational strength \cite{DNRWWKL,MR3629748} and this hierarchy interacts with Martin-L\"of random reals and completions of Peano Arithmetic \cite{MillerGreenberg,MR3251268}. The strongly non-recursive functions form an arguably at least as natural class, and here we start developing analogous hierarchy results for it. 

		\begin{definition}
			A function $f: \omega \rightarrow \omega$ is \emph{strongly nonrecursive} (or \emph{$\SNR$}) if for every recursive function $g$, for all but finitely many $n \in \omega$, $f(n) \neq g(n)$. It is \emph{strongly non-partial-recursive} (or \emph{$\SNPR$}) if for every partial recursive function $g$, for all but finitely many $n$, if $g(n)$ is defined, $f(n) \neq g(n)$.
		\end{definition}

		Note that every $\SNPR$ function $f$ is $\SNR$, as well as \emph{almost $\DNR$}: for all but finitely many $n$, if $\varphi_n(n)$ is defined, then $f(n) \neq \varphi_n(n)$. Also, a function is strongly nonrecursive iff it is \emph{eventually different} from each recursive function. Thus it is eventually different in the sense of set theory with the recursive sets as ground model \cite{MR1350295}.

		\begin{definition}
			An \emph{order function} is a recursive, nondecreasing, and unbounded function $h: \omega \rightarrow \omega$ such that $h(0) \ge 2$. For a class $\mathrm{C}$ of functions from $\omega$ to $\omega$, let $\mathrm{C}_h$ denote the subclass consisting of those members of $\mathrm{C}$ that are bounded by $h$.
		\end{definition}

		\begin{theorem}\label{rec-dnr-to-rec-snpr}
			For each order function $h$, there exists an order function $g$ such that every $\DNR_g$ function computes an $\SNPR_h$ function.
		\end{theorem}

		In order to prove this theorem, we will need a result due to Cenzer and Hinman \cite{CenzerHinman}, in a form presented in Greenberg and Miller \cite{MillerGreenberg}. 

		\begin{definition}[Greenberg and Miller \cite{MillerGreenberg}]
			Let $a \ge 2$ and let $c > 0$. Let $\mathcal{P}^c_a$ denote the class of functions $f$ bounded by $a$ such that for all $e$ and for all $x < c$, if $\varphi_e(x)\converges$, then $f(e) \neq \varphi_e(x)$.
		\end{definition}
			
		\begin{theorem}[Cenzer and Hinman \cite{CenzerHinman}]\label{avoiding-multiple}
			Let $a \ge 2$ and $c > 0$. Then any $\DNR_a$ function computes a function in $\mathcal{P}^c_{ca}$. Moreover, the reduction is uniform in $a$ and $c$.
		\end{theorem}

		\begin{proof}[Proof of Theorem~\ref{rec-dnr-to-rec-snpr}] Let $r$ be the recursive function such that $\varphi_{r(x)}(e) = \varphi_e(x)$. For each $n \ge 2$, let $x_n \in \omega$ be the least such that $h(x_n) \ge n^2$.

		We construct $g$ to ensure that any $\DNR_g$ function computes a function $f$ that is bounded by $h$ and such that for all $x > x_n$ and for all $e < n$, if $\varphi_e(x)\converges$, then $f(x) \neq \varphi_e(x)$.

		In order to compute $f$ on the interval $[x_n, x_{n+1})$, a function in $\mathcal{P}^n_{n^2}$ suffices and such a function can be uniformly obtained from a $\DNR_n$ function, by Theorem~\ref{avoiding-multiple}. However, we only need a finite part of this function, and we can recursively determine how much. Note that the reductions in Theorem~\ref{avoiding-multiple} can be assumed to be total. For $n \ge 2$, let $\gamma_n$ denote the use of reduction that, given a $\DNR_n$ function, computes a $\mathcal{P}^n_{n^2}$ function.

		Then, letting \[m_n = \max(\dset{\gamma_n(r(x))}{x \in [x_n, x_{n+1})]}),\] it suffices for $g$ to be any recursive function such that for all $x \le m_n$, $g(x) \le n$. \end{proof}

		We also have the following counterpart to Theorem~\ref{rec-dnr-to-rec-snpr}:
		\begin{theorem}\label{rec-dnr-to-rec-snrp-2}
			For each order function $g$, there exists an order function $h$ such that every $\DNR_g$ function computes an $\SNPR_h$ function. 
		\end{theorem}
		\begin{proof}
			For $n \ge 1$, let $\tau_n : \omega \rightarrow \omega^n$ be a uniformly recursive sequence of bijections, and for $i < n$,
			$\pi^n_i : \omega^n \rightarrow \omega$ denote the projection function onto the $i$-th coordinate.

			Let $r$ be the recursive function such that for $n \ge 1$ and $e < n$, if $\varphi_e(n)$ converges, then $\varphi_{r(e, n)}$, on any input, outputs $\pi^n_e(\tau_n(\varphi_e(n))$. 

			Now, given a $\DNR_g$ function $f$, let $j(0) = 0$ and for $n \ge 1$, let \[j(n) = \tau_n^{-1}(\langle f(r(0, n)), ..., f(r(n-1, n))\rangle).\]

	 		Then $j(n) \neq \varphi_e(n)$ for any $e < n$: If it were, then we would have
			\[
			f(r(e,n)) = \pi^n_e(\tau_n(j(n))) = \pi^n_e (\tau_n(\varphi_e(n))) = \varphi_{r(e,n)}(r(e,n)),
			\]
			which contradicts the fact that $f$ is $\DNR$.
			Thus, $j$ is $\SNPR_h$ where
			for $n \ge 1$,
			\[
				h(n)=\max(
					\{
						\tau_n^{-1}(\langle i_1, ..., i_{n-1}\rangle) : i_k < g(r(k, n)) \text{ for all $k < n$}
					\}
				).
			\]
		\end{proof}

		\begin{theorem}[Kjos-Hansen, Merkle, and Stephan \cite{MR2813422}]		
			Every non-high $\SNR$ is $\SNPR$.
		\end{theorem}
		\begin{proof}
			Supose that $f:\omega \rightarrow \omega$ is not high, and that $\psi$ is a partial recursive function that is infinitely often equal to it. For each $n \in \omega$, let $g(n)$ be the least stage such that $|\dset{x \in \omega}{\psi(x)[g(n)]\converges = f(x)}| \ge n + 1$. Then $g$ is recursive in $f$.

			Since $f$ is not high, there is a recursive function $h$ that escapes $g$ infinitely often. We define a recursive function $j$ that is infinitely often equal to $f$. Let $j_0 = \emptyset$. Given $j_n$, let \[A = \dset{\langle x, \psi(x)\rangle}{x \notin \dom(j_n), \psi(x)[h(n)] \converges}.\]

			Let $y$ be the least such that it is not in the domain of $j_n \cup A$. Finally, let $j_{n+1} = j_n \cup A \cup \langle y, 0 \rangle$. 

			It is easily seen that $j = \bigcup_n j_n$ is recursive and infinitely often equal to $f$. 
		\end{proof}

		\begin{corollary}\label{rec-snr-to-rec-dnr}
			Given any order function $h$, every non-high $\SNR_h$ function computes a $\DNR_h$ function.
		\end{corollary}

	\section{The SNR hierarchy}
		\subsection{Definitions and combinatorial lemmas}\label{sec:defs-and-combinatorics}

			The following definitions can also be found in \cite{MillerGreenberg} and \cite{KhanMiller}.

			\begin{definition}
				Given $\sigma \in \bstrings$, we say that a tree $T \subseteq \bstrings$ is \emph{$n$-bushy above $\sigma$} if
				every element of $T$ is comparable with $\sigma$, and
				for every $\tau \in T$ that extends $\sigma$ and is not a leaf of $T$, $\tau$ has at least $n$ immediate extensions in $T$.
				We refer to $\sigma$ as the \emph{stem} of $T$.
			\end{definition}

			\begin{definition}
				Given $\sigma \in \bstrings$, we say that a set $B \subseteq \bstrings$ is \emph{$n$-big above $\sigma$} if
				there is a finite $n$-bushy tree $T$ above $\sigma$ such that all its leaves are in $B$.
				If $B$ is not $n$-big above $\sigma$ then we say that $B$ is \emph{$n$-small} above $\sigma$.
			\end{definition}

			Proofs of the following lemmas can be found in \cite{MillerGreenberg} and \cite{KhanMiller}.

			\begin{lemma}[Smallness preservation property]\label{lem:big-subset}
				Suppose that $B$ and $C$ are subsets of $\bstrings$ and that $\sigma \in \bstrings$.
				If $B$ and $C$ are respectively $n$- and $m$-small above $\sigma$, then $B \cup C$ is $(n+m-1)$-small above $\sigma$.
			\end{lemma}

			\begin{lemma}[Small set closure property]\label{lem:small-set-closure}
				Suppose that $B \subset \bstrings$ is $n$-small above $\sigma$. Let $C = \dset{\tau \in \bstrings}{\text{$B$ is $n$-big above $\tau$}}$.
				Then $C$ is $n$-small above $\sigma$.
				Moreover $C$ is \emph{$n$-closed}, meaning that if $C$ is $n$-big above a string $\rho$, then $\rho \in C$.
			\end{lemma}

			\begin{definition}
				Given an order function $h$, let $h^{<\omega}$ denote the set of finite strings in $\omega^{<\omega}$ whose entries are bounded by $h$, and let $h^n$ denote the set of such strings of length $n$.
			\end{definition}

			\begin{theorem}\label{low-DNR-hierarchy} Let $h$ be any order function. Then, uniformly in $h$, we can find a recursive function $\pi$ such that if $g$ is any order function such that $h(n)/g(\pi(n))$ is unbounded, then there is a low $f \in \DNR_h$ that computes no $\DNR_g$ function.				
			\end{theorem}
			\begin{proof}
				Given $\sigma \in h^{<\omega}$, let $q(\sigma, e, k)$ be an index for the partial recursive function that searches for a $k$-big set $A \subset h^{<\omega}$ above $\sigma$ such that $\Phi^\tau_e(q(\sigma, e, k))$ converges and is constant as $\tau$ ranges over $A$, and which then outputs this constant value. Let \[\pi(n) = \max\dset{q(\sigma, e, k)}{\sigma \in h^n, e, k \le n}.\]

				Next, we describe a $0'$-recursive construction of $f$. We define a sequence \[f_0 \preceq f_1 \preceq f_2 \preceq ...\] of finite strings in $h^{<\omega}$, and \[B_0 \subseteq B_1 \subseteq B_2 \subseteq ... \] of r.e.\ subsets of $h^{<\omega}$ such that for each $s \in \omega$, $B_s$ is $h(|f_s|)$-small and $h(|f_s|)$-closed above $f_s$.

				Let $f_0 = \langle \rangle$, and let $B_0$ be the set of non-$\DNR$ strings. Next, we describe how to construct $f_{s+1}$ and $B_{s+1}$ given $f_s$ and $B_s$.

				\noindent \underline{If $s = 2e$ is even}: We ensure that $\Phi_e^f$ is not $\DNR_g$. Let $k = h(|f_s|)$ and let $n \ge k, e$ be the least such that $h(n) \ge k(g(\pi(n)) + 1)$. We begin by extending $f_s$ to a string $\sigma \notin B_s$ of length $n$. Note that $B_s$ is $k$-small and $k$-closed above $\sigma$. Let $x = q(\sigma, e, k)$, and note that $x \le \pi(n)$.

				Now, if $\varphi_{x}(x) \converges$ to some value $i$ less than $g(x)$, then the set 
				\[A_i = \dset{\tau \succeq \sigma}{\Phi^\tau_e(x)\converges = i}\] is $k$-big above $\sigma$, so there is an extension $\tau$ of $\sigma$ such that $\tau \in A_i \setminus B_s$. Let $f_{s+1} = \tau$ and $B_{s+1} = B_s$. This forces $\Phi^f_e$ to not be $\DNR$.

				Otherwise, for each $i < g(x)$, $A_i$ is $k$-small above $\sigma$, and so \[C = \bigcup_{i < g(x)} A_i\] is $ k g(x)$-small above $\sigma$, and $C \cup B_s$ is $k (g(x) + 1)$-small above $\sigma$. Since $h(n) \ge k (g(\pi(n)) + 1) \ge k (g(x) + 1)$, we can let $B_{s + 1} = B_s \cup C$ and $f_{s+1} = \sigma$. This forces $\Phi^f_e(x)$ to either diverge, or to converge to value greater than or equal to $g(x)$.

				\noindent \underline{If $s = 2e + 1$ is odd}: We ensure that $f$ is low. We begin by extending $f_s$ to a string $\sigma \notin B_s$ such that $h(|\sigma|) \ge 2 h(|f_s|)$. If the set \[F_e = \dset{\tau \succeq \sigma}{\varphi^\tau_e(e) \converges}\] is $h(|f_s|)$-big above $\sigma$, then there is a $\tau \in F_e \setminus B_s$. Let $f_{s+1} = \tau$ and $B_{s+1} = B_s$. This forces $e$ into the jump of $f$. Otherwise, let $f_{s+1} = \sigma$ and let $B_{s + 1} = B_s \cup F_e$, which is $h(|\sigma|)$-small above $\sigma$. This forces $e$ out of the jump of $f$. \end{proof}

			By making $g$ grow slowly enough, we get:
			\begin{corollary}\label{dnr-low-below}
				For every order function $h$ there is an order function $g$ such that there is a low $\DNR_h$ that computes no $\DNR_g$.
			\end{corollary}

			Additionally, we have:
			\begin{corollary}\label{dnr-low-above}
				For every order function $g$ there is an order function $h$ such that there is a low $\DNR_h$ that computes no $\DNR_g$.
			\end{corollary}

			\begin{proof} Using the uniformity in Theorem~\ref{low-DNR-hierarchy} along with the recursion theorem, we construct $h$ knowing its index in advance, thereby obtaining $\pi$, and ensuring that $h(n)/g(\pi(n))$ is unbounded. 
			\end{proof}

			By combining the strategies for the two corollaries above, we get:
		
			\begin{corollary}
				For every order function $h$ there is an order function $g$ such that there is a low $f_1 \in \DNR_g$ that computes no $\DNR_h$ as well as a low $f_2 \in \DNR_h$ that computes no $\DNR_g$.
			\end{corollary}

			\begin{corollary}\label{jan30-also}
				Given any order function $h$ there is an order function $g$ such that there is an $\SNR_h$ that computes no $\SNR_g$.
			\end{corollary}
			\begin{proof}
				By Theorem~\ref{rec-dnr-to-rec-snpr}, there is an order function $h'$ such that any $\DNR_{h'}$ computes an $\SNR_h$. By Corollary~\ref{dnr-low-below}, there is an order function $g$ such that there is a low $\DNR_{h'}$ function $f'$ that computes no $\DNR_g$ function. Then $f'$ computes an $\SNR_h$ function $f$ that computes no $\SNR_g$ function: if $j$ is recursive in $f$ and is an $\SNR_g$ function and since it is low, it is itself $\DNR_g$ by Corollary~\ref{rec-snr-to-rec-dnr}, a contradiction.
			\end{proof}

			\begin{corollary}\label{jan30-2018}
				Given any order function $g$ there is an order function $h$ such that there is an $\SNR_h$ that computes no $\SNR_g$.
			\end{corollary}
			\begin{proof}
				By Corollary~\ref{dnr-low-above}, there is an $h'$ and a low $\DNR_{h'}$ function $f'$ that computes no $\DNR_g$ function. By Theorem~\ref{rec-dnr-to-rec-snrp-2} there is an $h$ such that $f'$ computes an $\SNR_h$ function $f$. Then $f$ cannot compute an $\SNR_g$ function since the latter would be $\DNR_g$ by Corollary~\ref{rec-snr-to-rec-dnr}.
			\end{proof}

			Let $\mathfrak O$ denote the set of all order functions.
			Recall the Muchnik and Medvedev reducibilities of mass problems:
			\begin{definition}
				A mass problem $\mathcal A$ is Muchnik reducible to a mass problem $\mathcal B$, written $\mathcal A\le_w\mathcal B$ and sometimes read ``weakly reducible'', if for each $B\in\mathcal B$, there is an $A\in\mathcal A$ such that $A\le_T B$, where $\le_T$ is Turing reducibility. If there is a single Turing reduction $\Phi$ such that for all $B\in\mathcal B$, $\Phi^A\in\mathcal A$ then $\mathcal A$ is Medvedev reducible to $\mathcal B$, written $\mathcal A\le_S\mathcal B$ and sometimes read ``strongly reducible''.
			\end{definition}
			We can phrase Corollaries \ref{jan30-2018} and \ref{jan30-also} as follows:
			\[
				\forall h\in\mathfrak O\quad\exists g\in\mathfrak O\quad\SNR_g\not\le_w\SNR_h;
			\]
			\[
				\forall g\in\mathfrak O\quad\exists h\in\mathfrak O\quad\SNR_g\not\le_w\SNR_h.
			\]
			Thus, the Muchnik degrees of various mass problems $\SNR_f$ have no least or greatest element.

			\section{Canonical immunity}
			Canonical immunity was introduced by three of the present authors in \cite{MR3629746} and shown there to be equivalent, as a mass problem, to SNR, and studied further in \cite{2017arXiv170703947B}. Here we give a new Theorem \ref{bjoernFall2017} below, analogous to the case of DNR, that was not obtained in \cite{MR3629746}.
			
			Considering lowness notions associated with Schnorr randomness was what lead those authors to this new notion of immunity.

			\begin{definition}
			A \emph{canonical numbering of the finite sets} is a surjective function $D:\omega\rightarrow \{A:A\subseteq\omega\text{ and $A$ is finite}\}$ such that
			$\{(e,x): x\in D(e)\}$ is recursive and the cardinality function $e\mapsto |D(e)|$, or equivalently, $e\mapsto\max D(e)$, is also recursive.
		\end{definition}
			We write $D_e=D(e)$.
			\begin{definition}
				$R$ is \emph{canonically immune} ($\CI$) if $R$ is infinite and
				there is a recursive function $h$ such that for each canonical
				numbering of the finite sets $D_e$, $e\in\omega$,
				we have that for all but finitely many $e$, if $D_e\subseteq R$
				then $|D_e|\le h(e)$.
			\end{definition}
			We include proofs of some of the results from \cite{MR3629746}.
			\begin{theorem}[Beros, Khan, and Kjos-Hanssen {\cite{MR3629746}}]\label{lateApril}
				Schnorr randoms are canonically immune.
			\end{theorem}
			\begin{proof}
				Fix a canonical numbering of the finite sets, $\{D_e\}_{e \in \omega}$.
				Define $U_c = \{ X : (\exists e > c) \big( |D_e| \geq 2e \wedge D_e \subset X \big) \}$.
				Since $e \mapsto |D_e|$ is recursive, $\mu(U_c)$ is recursive and bounded by $2^{-c}$.
				Thus, the sequence $\{U_c\}_{c \in \omega}$ is a Schnorr test.
				If $A$ is a Schnorr random, then $A \in U_c$ for only finitely many $c \in \omega$.
				We conclude that $A$ is canonically immune.
			\end{proof}
			\begin{theorem}[Beros, Khan, and Kjos-Hanssen {\cite{MR3629746}}]
				Each canonically immune set is immune.
			\end{theorem}
			\begin{proof}
				Suppose $A$ has an infinite recursive subset $R$. Let $h$ be any recursive function.
				Let $R_n$ denote the set of the first $n$ elements of $R$, and let $\{D_e: e \in \omega\}$ be a canonical numbering of the finite sets such that $D_{2n} = R_{h(2n)+1}$ for all $n \in \omega$.
				For all $n$, $D_{2n}\subseteq R\subseteq A$ and $|D_{2n}| = h(2n)+1 > h(2n)$, and so $h$ does not witness the canonical immunity of $A$.
			\end{proof}

			We now show that canonically immune is the ``correct'' analogue
			of effectively immune. Let $W_0$, $W_1$, $W_2$, ... be an effective enumeration of the recursively enumerable (or r.e.) sets of natural numbers. An infinite set $A$ of natural numbers is said to be \emph{immune} if it contains no infinite r.e. subset. It is said to be \emph{effectively immune} when there is a recursive function $f$ such that for all $e$, if $W_e$ is a subset of $A$, then $|W_e| \le f(e)$. The interest in sets whose immunity is effectively witnessed in this manner originally arose in the search for a solution to Post's problem; for more on this the reader may see \cite{MR3629746}.

			\begin{theorem}[Beros, Khan, and Kjos-Hanssen {\cite{MR3629746}}]\label{may3}
				Each canonically immune ($\CI$) set computes a strongly nonrecursive
				function, i.e.,
				\[
					\SNR\le_w\CI.
				\]
			\end{theorem}
			Incidentally, Beros and Beros \cite{2016arXiv161001650B} showed that the index set of Medvedev reductions from $\CI$ to $\SNR$ is $\Pi^1_1$-complete.
			
			\begin{theorem}[{Kjos-Hanssen \cite{MR2813422}}]\label{bjoernCT}
				Each SNR function is either of high or DNR Turing degree.
			\end{theorem}
			\begin{corollary}[Beros, Khan, and Kjos-Hanssen {\cite{MR3629746}}]\label{borrowed}
				The following are equivalent for an oracle $A$:
				\begin{enumerate}
					\item $A$ computes a canonically immune set,
					\item $A$ computes an SNR function,
					\item $A$ computes an infinite subset of a Schnorr random.
				\end{enumerate}
			\end{corollary}
			If a canonically immune set $A$ is of non-high Turing degree then by Theorem \ref{may3}, $A$ computes an SNR function, which
			by Theorem \ref{bjoernCT} means that $A$ computes a DNR function, hence $A$ computes an effectively immune set. Our new result is to
			make this more direct: $A$ is itself that effectively immune set.

			\begin{theorem}\label{bjoernFall2017}
			If $A$ is non-high and canonically immune, then $A$ is effectively immune.
			\end{theorem}
			\begin{proof}
			Let us introduce the notation
			\[W_{e,s}\#u\]
			to mean the set of $k\in W_{e,s}$ such that when $k$ enters $W_{e,s}$, at most $u$ other numbers have already entered $W_{e,s}$. Roughly speaking, $D\# u$ consists of the first $u+1$ elements of $D$.

			Let $A$ be non-high and not effectively immune. We need to show that $A$ is not canonically immune.

			Since $A$ is not effectively immune, there are infinitely many $e$ for which there is an $s$ with $\la e,s\ra\in\mathcal F$
			where
			\[
				\mathcal F=\{\la e,s\ra : |W_{e,s}|>h(e), \text{and }W_{e,s}\#h(e)\subseteq A\}.
			\]
			Here we assume $h$ is nondecreasing.
			So
			\[
				\forall d\quad\exists s\quad\exists e_d>d\quad \la e_d,s\ra\in\mathcal F
			\]
			Let $f(d)=s$. Then there is a recursive function $g(d)$ which is not dominated by $f$.

			We may assume $s>e_d$ since increasing $s$ will keep $\la e,s\ra\in\mathcal F$.
			Let
			\[D_{2\la e,d\ra}=W_{e,g(d)}\#h(d)\]
			Let $\tilde h(\la e,d\ra)=h(d)$ for $d\le e\le g(d)$. Using a recursive bijection between the domain of $\tilde h$,
			\[
			\{\la e,d\ra : d\le e\le g(d)\}
			\]
			and $\omega$, we may assume the domain of $\tilde h$ is $\omega$.

			Then for infinitely many $e$ (namely, there are infinitely many $d$ with $f(d)\le g(d)$, and for each such $d$ there is an $e$ with $d\le e\le g(d)$ that works) we have
			\begin{enumerate}
			\item $D_{2\la e,d\ra}=W_{e,g(d)}\#h(d)\subseteq W_{e,g(d)}\#h(e)\subseteq A$, and
			\item $|D_{2\la e,d\ra}|>\tilde h(\la e,d\ra)=h(d)$.
			\end{enumerate}

			We have $\lim_{n\to\infty}\tilde h(n)=\infty$ since for each $d$ we only include $e$ up to the $e_d$ above.

			Let the sets $D_{2k+1}$ be a canonical list of all the finite sets, just in case we missed some of them using the sets $D_{2k}$.
			\end{proof}

	\section{Pandemic numberings}

		Brendle et al.~\cite{MR3445411} explored an analogy between the theory of
		cardinal characteristics in set theory and highness properties in
		computability theory, following up on work of Rupprecht \cite{MR2660926}.

		Their main results concerned a version of Cichon's diagram \cite{MR1233917} in computability
		theory. They expressed the cardinal characteristics in Cichon's diagram
		as either
		\[
			\mathfrak{d}(R) = \min\{|F|: F\subseteq Y\text{ and }\forall x \in X\,\exists y \in F\, (x \mathrel{R} y)\},
		\]
		\[
			\mathfrak{d}(R) = \min\{|F|: F\subseteq Y\text{ and }\forall x \in X\,\exists y \in F\, (x \mathrel{R} y)\}
		\]
		or
		\[
			\mathfrak{b}(R) = \min\{|G|: G\subseteq X\text{ and }\forall y \in Y\,\exists x \in G\, \neg (x \mathrel{R} y)\},
		\]
		\[
			\mathfrak{b}(R) = \min\{|G|: G\subseteq X\text{ and }\forall y \in Y\,\exists x \in G\, \neg (x \mathrel{R} y)\},
		\]
		where $X$ and $Y$ are two spaces and $R$ is a relation on $X \times Y$.
		As the spaces considered admit a notion of
		relative computability, it is natural to say that an element $x\in X$ is computable (in $A \subset \omega$).

		They defined two computability-theoretic notions corresponding to
		$\mathfrak{d}(R)$ and $\mathfrak{b}(R)$ as follows:
		\[
			\mathcal{B}(R) = \{A: \exists y \leq_T A\, \forall x\, (x \text{ is	computable } \to x \mathrel{R} y)\},
		\]
		\[
			\mathcal{B}(R) = \{A: \exists y \leq_T A\, \forall x\, (x \text{ is	computable } \to x \mathrel{R} y)\},
		\]
		and
		\[
			\mathcal{D}(R) = \{A: \exists x \leq_T A\, \forall y\, (y \text{ is	computable } \to \neg(x \mathrel{R} y))\},
		\]
		\[
			\mathcal{D}(R) = \{A: \exists x \leq_T A\, \forall y\, (y \text{ is	computable } \to \neg(x \mathrel{R} y))\}.
		\]
		They found that $\mathcal{B}(R)$ and $\mathcal{D}(R)$ tend to be highness properties in computability theory, often
		equivalent to well-known notions.

		They finally mapped a cardinal characteristic $\mathfrak{d}(R)$ or $\mathfrak{b}(R)$ to
		$\mathcal{D}(R)$ or $\mathcal{B}(R)$ respectively, and showed that
		if $\mathfrak{b}(R) \leq \mathfrak{d}(S)$, then $\mathcal{B}(R) \subseteq
		\mathcal{D}(S)$ and so on.

		The resulting analog of Cichon's diagram is not isomorphic
		to the original diagram, in that some strict inequalities of cardinal
		characteristics are consistent with ZFC but the corresponding
		computability-theoretic notions coincide.

		Using their point of view we find a new dual notion to canonical immunity: that of \emph{a numbering such that no recursive set is large with respect to it}.
		\begin{definition}
			Let $D=(e\mapsto D_e)$ be a numbering of the finite subsets of $\omega$ and let $R\subseteq\omega$.
			We say that $D$ is \emph{$h$-endemic to $R$} if there are infinitely many $e$ with $|D_e|\ge h(e)$ and $D_e\subseteq R$.
			$D$ is a \emph{pandemic numbering} if there is an order function $h$ such that for all infinite recursive sets $R$, $D$ is $h$-endemic to $R$.
		\end{definition}
		Recall also that an \emph{escaping function} is a function $f:\omega\to\omega$ such that $f$ is not dominated by any recursive function.
		\begin{theorem}\label{bjoern2018}
			The Muchnik degrees of pandemic numberings and of hyperimmune sets coincide.
		\end{theorem}
		\begin{proof}
			In one direction, if $e\mapsto\max D_e$ is recursively bounded then we show that $D$ is not a pandemic numbering.
			Namely, we construct a recursive set $R$ which waits for $h$ to get large ($h(e)> d$ say) and only then lets its $d$th element $r_d$ enter $R$, and lets $r_d$ be large enough (larger than $\max D_k$, $k\le e$) to prevent $|D_e|\ge h(e)$, $D_e\subseteq R$.

			In the other direction, given a hyperimmune set, we (as is well known) also have an escaping function $f$.
			At those inputs $e$ where $f$ is greater than a function associated with a potential infinite recursive set $R_k$ (given as the graph of a partial recursive $\{0,1\}$-valued function),
			we define the next $D_e$ so as to ensure $D_e\subseteq R$ and $|D_e|\ge h(e)$.
			Namely, the function to escape is the time $e\mapsto t_k(e) = t(\la e,k\ra)$ it takes for $R_k$ to get $h(e)$ many elements.
			If $R_k$ is really infinite then $t_k$ is total and so $f(e)\ge t_k(e)$ for some (infinitely many) $e$.
			So we define $D_{\la e,k\ra}$ to consist of the first $h(\la e,k\ra)$ elements of $R_k$, if any, as found during a search time of $f(\la e,k\ra)$ or even just $f(e)$.
		\end{proof}
		By Corollary \ref{borrowed} and Theorem \ref{bjoern2018},
		the dualism between immunity and pandemics is the same, Muchnik-degree-wise, as that between, in the notation of Brendle et al.~\cite{MR3445411}, 
		\begin{itemize}
			\item eventually different functions, $\mathfrak b(\ne^*)$, and
			\item functions that are infinitely often equal to each recursive function, $\mathfrak d(\ne^*)$.
		\end{itemize}

		\begin{remark}
			As the recursive sets are closed under complement, a notion of \emph{bi-pandemic} would be the same as \emph{pandemic}. Thus, while we do not know whether canonical bi-immunity is Muchnik equivalent to canonical immunity, there is no corresponding open problem on the dual side.
		\end{remark}

	\bibliography{research,bibliography}{}
	\bibliographystyle{plain}
\end{document}